\documentclass[12pt]{article}
\usepackage{amssymb}
 \usepackage{mathtools}
 \usepackage{pifont}
 \usepackage{multicol}
\usepackage{mathrsfs}
\usepackage{color}
\usepackage{multicol}
\usepackage{amsthm}
\usepackage{float}
\usepackage{tikz}
\usepackage{tikz-cd}
\usetikzlibrary{arrows}
\usetikzlibrary{shapes.arrows}   
\usetikzlibrary{positioning}
\usepackage{mathrsfs}
\usepackage{amssymb}
\usepackage{mathtools}

\DeclarePairedDelimiter{\ceil}{\lceil}{\rceil}
\definecolor{1}{rgb}{1,0.2,0.3}
\definecolor{2}{rgb}{0.1,0.3,0.5}
\definecolor{3}{rgb}{1,1,0}
\definecolor{4}{rgb}{255,255,255}
\newtheorem{theorem}{Theorem}[section]
\newtheorem{corollary}{Corollary}[theorem]
\newtheorem{lemma}[theorem]{Lemma}
\newtheorem{conjecture}[theorem]{Conjecture}
\theoremstyle{definition}

\theoremstyle{remark}

\begin{document}

\tikzset
{
  x=.23in,% default value is 1cm.
  y=.23in,
}

\title{Polyominoes with maximally many holes\\}
\author{Matthew Kahle and \'Erika Rold\'an}
\date{\today}
\maketitle

\begin{abstract}
What is the maximum number of holes that a polyomino with $n$ tiles can enclose? Call this number $f(n)$. We show that if $n_k = \left( 2^{2k+1} + 3 \cdot 2^{k+1}+4 \right) / 3$ and $h_k =  \left( 2^{2k}-1 \right) /3$, then $f(n_k) = h_k$ for $k \ge 1$.
We also give nearly matching upper and lower bounds for large $n$, showing as a corollary that $f(n) \approx n/2$.
%\[\frac{1}{2}n - \sqrt{ \left( \frac{5}{2}-o(1)  \right) n} \leq f(n)\leq \frac{1}{2}n - \sqrt{\frac{3}{2}n}+o \left(\sqrt{n} \right).\]
This paper is dedicated to the memory of Solomon W. Golomb.
\end{abstract}
\section{Polyominoes.}
%We study the maximum number of holes that polyominoes of a given area can have. 
\emph{Polyominoes}, first studied systematically by Golomb \cite{golomb1954checker}, are shapes that can be made by gluing together finitely many unit squares, edge to edge. A polyomino with $n$ squares is sometimes called an \emph{$n$-omino}.

Tiling problems involving polyominoes are well studied---see, for example, \cite{MR1291821} or \cite{boltjanskij1991geometric}. In tiling problems, one almost always restricts to simply-connected polyominoes, i.e.,\ polyominoes without holes. Our main interest here is maximizing the number of holes. Figure \ref{Fonehole} illustrates an $8$-omino, a $20$-omino, and a $60$-omino with $1$, $5$, and $21$ holes respectively.

To be precise about the topology, we consider the tiles to be closed unit squares in the plane. Polyominoes are finite unions of these closed squares, so they are compact. The holes of a polyomino are defined to be the bounded connected components of its complement in the plane.

%%%%%
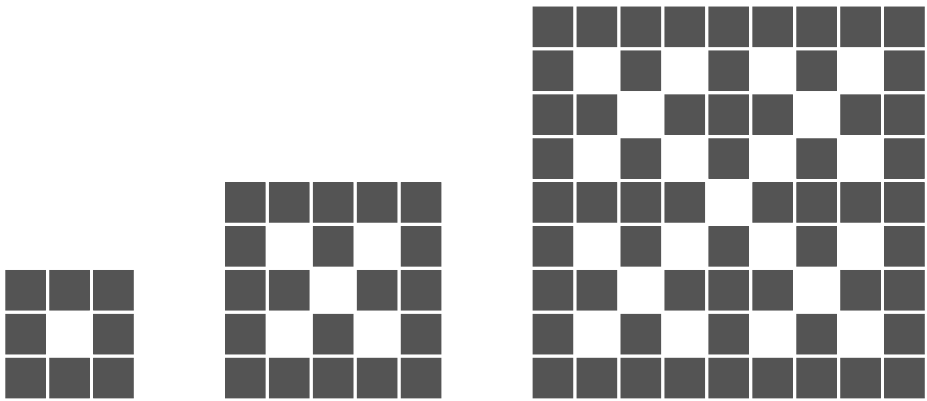
\begin{figure}[ht] 
\begin{tikzpicture} 
\foreach \x/\y in {0 / 0, 0 / 1, 1 / 0, 2 / 0, 2 / 1, 0 / 2, 1 / 2, 2 / 2 } { 
\path [draw=gray!66!black, fill=gray!66!black] (.5+\x-0.45, .5+\y-0.45)
-- ++(0,.9)
-- ++(.9,0)
-- ++(0,-.9)
--cycle;
}

\foreach \x/\y in {0 / 0, 0 / 1, 1 / 0, 2 / 0, 2 / 1, 0 / 2, 1 / 2, 4 / 0, 4 / 1, 3 / 0, 4 / 2, 3 / 2, 0 / 4, 0 / 3, 1 / 4, 2 / 4, 2 / 3, 4 / 4, 4 / 3, 3 / 4 } { 
\path [draw=gray!66!black, fill=gray!66!black] (2+3.5+\x-0.45, .5+\y-0.45)
-- ++(0,.9)
-- ++(.9,0)
-- ++(0,-.9)
--cycle;
}

\foreach \x/\y in {0 / 0, 0 / 1, 1 / 0, 2 / 0, 2 / 1, 0 / 2, 1 / 2, 4 / 0, 4 / 1, 3 / 0, 4 / 2, 3 / 2, 0 / 4, 0 / 3, 1 / 4, 2 / 4, 2 / 3, 4 / 3, 3 / 4, 8 / 0, 8 / 1, 7 / 0, 6 / 0, 6 / 1, 8 / 2, 7 / 2, 5 / 0, 5 / 2, 8 / 4, 8 / 3, 7 / 4, 6 / 4, 6 / 3, 5 / 4, 0 / 8, 0 / 7, 1 / 8, 2 / 8, 2 / 7, 0 / 6, 1 / 6, 4 / 8, 4 / 7, 3 / 8, 4 / 6, 3 / 6, 0 / 5, 2 / 5, 4 / 5, 8 / 8, 8 / 7, 7 / 8, 6 / 8, 6 / 7, 8 / 6, 7 / 6, 5 / 8, 5 / 6, 8 / 5, 6 / 5 } { 
\path [draw=gray!66!black, fill=gray!66!black] (4+ 8.5+\x-0.45, .5+\y-0.45)
-- ++(0,.9)
-- ++(.9,0)
-- ++(0,-.9)
--cycle;
}

 \end{tikzpicture}

\caption{Polyominoes with holes.}
\label{Fonehole}
\end{figure}
%%%%%

\subsection{Counting Holes.}
We consider two $n$-ominoes to be equivalent if they agree after a translation. We denote the set of all $n$-ominoes by $\mathcal{A}_{n}$.
Given a polyomino $A$, we denote by $h(A)$ the number of holes in $A$. For $n\ge 1$ define,
\begin{equation}
f(n):= \max_{A\in \mathcal{A}_n}h(A).
\end{equation}

Other notions of polyomino equivalence have been studied, especially in the context of polyomino enumeration. For example, one might consider two polyominoes to be the same if they agree after rotation or reflection. However, the number of holes is invariant under rigid motions, so for our purposes, all that matters is that $\mathcal{A}_n$ is a finite set so the maximum $f(n)$ is well defined. This maximum is the same under any of these notions of equivalence.

Similarly, define $g(m)$ to be the minimum number $N$ such that there exists an $N$-omino with $m$ holes. We check below that $g$ is a right inverse of $f$.

The function $g$ is listed at The On-Line Encyclopedia Of Integer Sequence as sequence A118797. 
Tom\'as Oliveira e Silva enumerated free polyominoes according  to area and number of holes, up to $n=28$, which at the time of this writing seems to be the state of the art. As a corollary of his calculations, we know $f(n)$ for $n \le 28$ and $g(m)$ for $m \le 8$. \cite{bworld}. See Table 1.

\begin{table}
\centering
\begin{tabular}{l | l | l | l | l | l | l | l | l}
m & 1 & 2  & 3  & 4  & 5  & 6  & 7  & 8  \\ \hline
g(m) & 7 & 11 & 14 & 17 & 19 & 23 & 25 & 28 \\
\end{tabular}
\caption{$g(m)$ for $m \le 8$.}
\end{table}

\subsection{Statement of main results.}

Here and throughout, 
\begin{equation}\label{Defn}
n_k=\frac{1}{3} \left( 2^{2k+1}+3 \cdot 2^{k+1}+4 \right),
\end{equation}
and 
\begin{equation}\label{Defh}
h_k=\frac{1}{3} \left(2^{2k}-1 \right).
\end{equation}

\begin{theorem}\label{Ttremebundoexactly}
For $k \ge 1$ we have 
$f(n_k)=h_k$. Moreover,
$f(n_k - 1)=h_k$ and $f(n_k -2) =h_k-1$.
\end{theorem}

%The following result gives the exact values of $g(m)$ for infinitely values of $m$. It is a corollary of Theorem \ref{Ttremebundoexactly} and the properties of $f$ and $g$ that we prove in Section \ref{Spreliminaryr}.

\begin{corollary}\label{aboutexactg}
$g(h_k) = n_k-1$ for all $k \ge 1$.\\
\end{corollary}

We also give bounds for large $n$.

\begin{theorem}\label{Ttremebundo}
Let $f(n)$ denote the maximum number of holes that a polyomino with $n$ squares can have. Given $C_{1}>\sqrt{5 / 2}$ and $C_{2}<\sqrt{3 / 2}$, there exists an $n_0=n_0(C_1,C_2)$ such that 
\begin{equation*}
\frac{1}{2}n-C_{1}\sqrt{n} \leq f(n) \leq \frac{1}{2}n-C_{2}\sqrt{n},
\end{equation*}
for $n>n_0$.
\end{theorem}

Theorem \ref{Ttremebundoexactly} implies that the constant $\sqrt{3/2}$ in Theorem \ref{Ttremebundo} can not be improved.

\section{Preliminary Results.}\label{Spreliminaryr}

%In this section we study properties of the functios $f$ and $g$ and establish how do they relate to one another.

It is clearly always possible to attach a square tile to an $n$-omino and obtain an $(n+1)$-omino with at least the same number of holes. This implies that $f(n)\leq f(n+1)$ for evey $n \ge 1$, so $f$ is monotonically increasing. The following lemma tells us that $f$ never increases by more than one.

\begin{lemma}\label{Ljumpsoff}
For every $n \ge 1$, $$ f(n+1)-f(n)\leq 1.$$
\end{lemma}
\begin{proof}
We will show that if $A$ is any $(n+1)$-omino, then there exists an $n$-omino $B$ such that $h(B) \ge h(A)-1$.

Let $A$ be an $(n+1)$-omino, and $k$ be the number of tiles in the bottom row of $A$, and denote by $l$ the leftmost tile in this row.% We prove by induction on $k$ that it is possible to delete a tile from $A$ in such a way that an $n$-omino is created with the same holes as $A$ or with one less hole.

If $k=1$, then $l$ is only connected to one other tile, so we can delete $l$ without disconnecting $A$ or destroying any holes.

Now suppose the statement holds whenever $k =1, 2, \dots,  m$. Then let $k=m+1$ and denote by $l_1$, $l_2$, and $l_3$ the tile sites that share boundary with $l$ that are to the up and right to $l$.
Each of the tile sites $l_1$, $l_2$, and $l_3$ could either be occupied by tiles in $A$ or not. The six possibilities are depicted in Figure \ref{Flcomb2}. Any other combination would result in $A$ having disconnected interior.

%%%%
%\input{l.txt}
 %%%%

%%%%
%\input{lcomb.txt}
 %%%%

%%%%
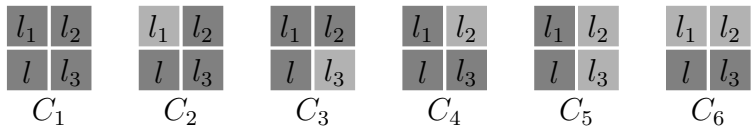
\begin{figure}[ht]
\begin{center}
\begin{tikzpicture} 
\foreach \x/\y in {0 / 0, 0 / 1, 1 / 0, 1 / 1} { 
\path [draw=gray, fill=gray] (.5+\x-0.45, .5+\y-0.45)
-- ++(0,.9)
-- ++(.9,0)
-- ++(0,-.9)
--cycle;
}

\draw (.5,.45) node {$l$};
\draw (.5,1.45) node {$l_1$};
\draw (1.5,.45) node {$l_3$};
\draw (1.5,1.45) node {$l_2$};

\draw (1,-.45) node {$C_1$};

\foreach \x/\y in {0 / 0,  1 / 0, 1 / 1} { 
\path [draw=gray, fill=gray] (1+2.5+\x-0.45, .5+\y-0.45)
-- ++(0,.9)
-- ++(.9,0)
-- ++(0,-.9)
--cycle;
}

\foreach \x/\y in { 0 / 1} { 
\path [draw=lightgray!93!black, fill=lightgray!93!black] (1+2.5+\x-0.45, .5+\y-0.45)
-- ++(0,.9)
-- ++(.9,0)
-- ++(0,-.9)
--cycle;
}

\draw (3.5,.45) node {$l$};
\draw (3.5,1.45) node {$l_1$};
\draw (4.5,.45) node {$l_3$};
\draw (4.5,1.45) node {$l_2$};

\draw (4,-.45) node {$C_2$};

\foreach \x/\y in { 0 / 0, 0 / 1, 1 / 1 } { 
\path [draw=gray, fill=gray] (2+ 4.5+\x-0.45, .5+\y-0.45)
-- ++(0,.9)
-- ++(.9,0)
-- ++(0,-.9)
--cycle;
}

\foreach \x/\y in {1 / 0} { 
\path [draw=lightgray!93!black, fill=lightgray!93!black] (2+4.5+\x-0.45, .5+\y-0.45)
-- ++(0,.9)
-- ++(.9,0)
-- ++(0,-.9)
--cycle;
}

\draw (6.5,.45) node {$l$};
\draw (6.5,1.45) node {$l_1$};
\draw (7.5,.45) node {$l_3$};
\draw (7.5,1.45) node {$l_2$};

\draw (7,-.45) node {$C_3$};

\foreach \x/\y in { 0 / 0, 0 / 1, 1 / 0 } { 
\path [draw=gray, fill=gray] (3+ 6.5+\x-0.45, .5+\y-0.45)
-- ++(0,.9)
-- ++(.9,0)
-- ++(0,-.9)
--cycle;
}

\foreach \x/\y in { 1 / 1} { 
\path [draw=lightgray!93!black, fill=lightgray!93!black] (3+6.5+\x-0.45, .5+\y-0.45)
-- ++(0,.9)
-- ++(.9,0)
-- ++(0,-.9)
--cycle;
}

\draw (9.5,.45) node {$l$};
\draw (9.5,1.45) node {$l_1$};
\draw (10.5,.45) node {$l_3$};
\draw (10.5,1.45) node {$l_2$};

\draw (10,-.45) node {$C_4$};

\foreach \x/\y in { 0 / 0, 0 / 1 } { 
\path [draw=gray, fill=gray] (4+ 8.5+\x-0.45, .5+\y-0.45)
-- ++(0,.9)
-- ++(.9,0)
-- ++(0,-.9)
--cycle;
}

\foreach \x/\y in { 1 / 0, 1 / 1} { 
\path [draw=lightgray!93!black, fill=lightgray!93!black] (4+8.5+\x-0.45, .5+\y-0.45)
-- ++(0,.9)
-- ++(.9,0)
-- ++(0,-.9)
--cycle;
}

\draw (12.5,.45) node {$l$};
\draw (12.5,1.45) node {$l_1$};
\draw (13.5,.45) node {$l_3$};
\draw (13.5,1.45) node {$l_2$};

\draw (13,-.45) node {$C_5$};

\foreach \x/\y in { 0 / 0,  1 / 0} { 
\path [draw=gray, fill=gray] (5+ 10.5+\x-0.45, .5+\y-0.45)
-- ++(0,.9)
-- ++(.9,0)
-- ++(0,-.9)
--cycle;
}

\foreach \x/\y in { 0 / 1, 1 / 1} { 
\path [draw=lightgray!93!black, fill=lightgray!93!black] (5+10.5+\x-0.45, .5+\y-0.45)
-- ++(0,.9)
-- ++(.9,0)
-- ++(0,-.9)
--cycle;
}

\draw (15.5,.45) node {$l$};
\draw (15.5,1.45) node {$l_1$};
\draw (16.5,.45) node {$l_3$};
\draw (16.5,1.45) node {$l_2$};

\draw (16,-.45) node {$C_6$};

 \end{tikzpicture}
\end{center}
\caption{The tile $l$ denotes the leftmost tile in the bottom row of a polyomino $A$. If the tile sites $l_1$, $l_2$, or $l_3$ are in $A$, they are dark gray colored, otherwise they are colored with the lightest gray. $C_1$-$C_6$ are the six possible combinations for these tile sites. All other possibilities are rejected because they give a square structure with a non connected interior and we are supposing that $A$ is a polyomino.}
\label{Flcomb2}
\end{figure}
 %%%%

If $A$ is such that the local tile structure around $l$ coincides with $C_1$, $C_2$, $C_3$, $C_5$, or $C_6$, then it is already possible to delete $l$ from $A$ to generate an $n$-omino $B$ such that $h(B) \geq h(A)-1$. 

However, if $C_4$ is the local structure around $l$, then it is possible that deleting $l$ disconnects $A$. In this case, we delete $l$ and then add a new tile at the empty tile site $l_2$. This yields a new polyomino $A^{\prime}$ with the same number of tiles. If the addition of this new tile causes the coverage of a hole, then $h(A^\prime)= h(A)-1$, and $C_1$ or $C_3$ must then be the new local structure around the leftmost tile of the bottom row of $A^\prime$. This then allows us to terminate the process by deleting the bottom leftmost tile from $A^\prime$ without destroying more holes. If we have not destroyed any holes, then we have an $(n+1)$-omino $A^\prime$ with $h(A)=h(A^\prime)$ and with $k=m$ tiles in the bottom row. Hence, we can apply the induction hypothesis and the desired result follows. 
\end{proof}

%\begin{corollary}\label{Cjumpsoff}
%For every $n \ge 1$, $$ f(n+1)-f(n)\leq 1.$$
%\end{corollary}
%\begin{proof}
%Let $A$ be an $(n+1)$-omino such that $h(A)=f(n+1)$. By Lemma \ref{Ljumpsoff} there exists a polyomino $B$ with $n$ tiles such that $h(B)\geq h(A)-1$. This implies that $f(n)\geq f(n+1)-1$. Then, from the fact that $f$ is an increasing function, we can conclude that $ f(n+1)-f(n)\leq 1$. 
%\end{proof}

\begin{lemma}\label{Lfandg}
For every $m \ge 1$, we have that $f(g(m))=m$. Also, $g(m) = n$ if and only if $f(n)=m$ and $f(n-1) = m-1$.
\end{lemma}
\begin{proof} 
By the definitions of $f$ and $g$, we have immediately that $f \left( g(m) \right) \geq m $. Suppose by way of contradiction that for some $m$ we have $f(g(m)) \ge m+1$. By Lemma \ref{Ljumpsoff}, $f(g(m)-1)\geq m$. This implies that there exists a polyomino with $g(m)-1$ tiles and at least $m$ holes, but this contradicts the definition of $g$. We then conclude that $f ( g(m) ) = m $ for every $m$. It follows immediately that $g(m) = n$ if and only if $f(n)=m$ and $f(n-1) = m-1$.
\end{proof}

\section{An Upper Bound.}
\subsection{Perimeter.}
Define the perimeter $p(A)$ of a polyomino $A \in \mathcal{A}_n$ as the number of edges in the topological boundary of $A$. That is, $p(A)$ counts the number of edges with a square in $A$ on one side, and a square not in $A$ on the other side. For example, the 8-omino in Figure \ref{Fonehole} has a perimeter equal to 16. 

For $n \ge 1$ we denote $p_{\min}(n)$ and $p_{\max}(n)$ as the minimum perimeter and the maximum perimeter possible for a polyomino with an area of $n$.

In 1976, F. Harary and H. Harborth \cite{harary1976extremal} proved that the minimum perimeter possible in an  $n$-omino is given by, 
\begin{equation}\label{perimeter1976}
p_{min}(n)=2\ceil{2\sqrt{n}}.\\ \\
\end{equation}

%
%In particular, they exhibited a family of spiral-like polyominoes that achieve the minimum perimeter. Any square polyomino is a spiral-like polyomino (see Figure 3). In \cite{kurz2008counting} S. Kurz counts the number of polyominoes (up to rotations, translations and reflections) that have perimeter equal to the minimum perimeter. %We are nor aware of published results concerning the enumeration of polyominoes that have maximum perimeter.

%%%%
\begin{figure}[h]
\begin{center}
\begin{tikzpicture} 
\foreach \x/\y in {0 / 0, 0 / 1, 1 / 0, 1 / 1, 1/2, 2/0, 2/1 } { 
\path [draw=gray!66!black, fill=gray!66!black] (.5+\x-0.45, .5+\y-0.45)
-- ++(0,.9)
-- ++(.9,0)
-- ++(0,-.9)
--cycle;
}

\foreach \x/\y in {0 / 0, 0 / 1, 1 / 0, 1 / 1, 0/2, 2/0, 2/1} { 
\path [draw=gray!66!black, fill=gray!66!black] (2+3.5+\x-0.45, .5+\y-0.45)
-- ++(0,.9)
-- ++(.9,0)
-- ++(0,-.9)
--cycle;
}

\foreach \x/\y in { 0 / 1, 0/2, 1 / 0, 1 / 1, 1/2, 2/1, 2/0 } { 
\path [draw=gray!66!black, fill=gray!66!black] (4+ 6.5+\x-0.45, .5+\y-0.45)
-- ++(0,.9)
-- ++(.9,0)
-- ++(0,-.9)
--cycle;
}

 \end{tikzpicture}
\end{center}
\caption{Polyominoes that achieve the minimum perimeter.}
\label{maxper}
\end{figure}
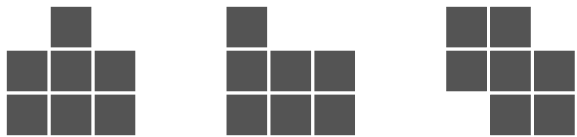

 Let $A \in \mathcal{A}_{n}$, then the number of edges that are on the boundary of two squares of $A$ will be denoted by $b(A)$,  in which case the edges are contained in the interior of the polyomino. Observe that all the edges of the squares of a polyomino either belong to the perimeter or to the interior of the polyomino. This means that,
 \begin{equation}\label{perimeterequal}
 4n= p(A)+ 2 b(A).
 \end{equation}

For example, if $A$ is any 7-omino depicted in Figure \ref{wormlike}, then $b(A)=n-1=6$ and $p(A)=4n-2(n-1)=16$.

Let $b_{min}(n)$ be the minimum number of edges shared by two squares that an $n$-omino can have. It is possible to associate a dual graph with any polyomino by considering each square as a vertex and by connecting any two of these vertices if they share an edge.

%%%%%
\begin{figure}[ht]
\centering

\begin{tikzpicture} 

\foreach \x/\y in {0 / 3, 1 / 2, 1 / 3, 2 / 1, 2/2, 3/0, 3/1 } { 
\path [draw=gray!66!black, fill=gray!66!black] (.5+\x-0.45, .5+\y-0.45)
-- ++(0,.9)
-- ++(.9,0)
-- ++(0,-.9)
--cycle;
}

\foreach \x/\y in {0/1, 0/3, 1/1, 1/2, 1/3, 1/0, 2,2} { 
\path [draw=gray!66!black, fill=gray!66!black] (2+3.5+\x-0.45, .5+\y-0.45)
-- ++(0,.9)
-- ++(.9,0)
-- ++(0,-.9)
--cycle;
}

\foreach \x/\y in {0 / 0, 0 / 1, 1 / 0, 2 / 0, 2 / 1, 0 / 2, 1 / 2 } { 
\path [draw=gray!66!black, fill=gray!66!black] (4+ 3.5+2+\x-0.45, .5+\y-0.45)
-- ++(0,.9)
-- ++(.9,0)
-- ++(0,-.9)
--cycle;
}

 \end{tikzpicture}

\caption{Polyominoes that achieve the maximum perimeter.}

\label{wormlike}

\end{figure}
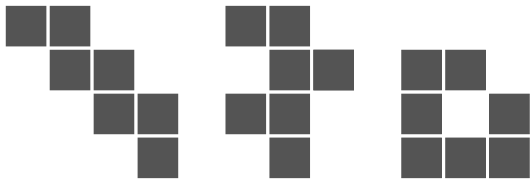
%%%%%

 Let $A$ be an $n$-omino. Because any polyomino has to have a connected interior, then the associated dual graph of $A$ is connected and there exists a spanning tree of this graph. If the dual graph has $n$ vertices, then there are at least $n-1$ edges in the spanning tree. This implies that there are at least $n-1$ different common edges in $A$. That is, $b(A)\geq n-1$, and this is true for any $n$-omino. This then gives us that $b_{\min}(n)\geq n-1$. Observing that the $n$-omino $C$ with only one column has $b(C)=n-1$, we conclude that $b_{min}(n)=n-1$. 

As a consequence of this and equality (\ref{perimeterequal}), for every $n\geq 1$, 
\begin{equation}
p_{max}(n)\leq 4n-2b_{min}(n)=4n-2(n-1)=2n+2,
\end{equation}
and equality is only achieved by polyominoes with the number of common edges equal to $b_{min}(n)$.

We need to distinguish when an edge that is on the perimeter of a polyomino $A \in \mathcal{A}_{n}$ is an edge that forms part of a hole in the polyomino. Define such edges as being part of the hole perimeter. We represent by $p_{h}(A)$ the number of edges on the hole perimeter of $A$. We define the outer perimeter of $A$, denoted by $p_{o}(A)$, as the difference between the perimeter $p(A)$ and the hole perimeter,
\[p_{o}(A)=p(A)-p_{h}(A).\]

If a polyomino $A$ is simply connected,
 then $p(A)=p_{o}(A)$. In general, $p(A)=p_{o}(A)+p_{h}(A)$ by definition.

Polyominoes with holes might achieve the maximum perimeter. However, the next lemma checks the intuitive fact that the minimum perimeter cannot be achieved by polyominoes with holes.
\begin{lemma}\label{nop}
If $A \in \mathcal{A}_{n}$ and $A$ has at least one hole, then $p_{min}(n)<p(A)$.
\end{lemma}
\begin{proof}
By (\ref{perimeter1976}), an $n$-omino with $k$ holes, with $k>1$,  has an outer perimeter at least equal to $p_{min}(k+n)$. A polyomino with $k$ holes has a hole perimeter greater or equal to $4k$. Then, we have,
\begin{equation}\label{holesnominp}
p(A)=p_{o}(A)+p_{h}(A)\geq p_{min}(n+k)+4k.
\end{equation}
Because the function $h(x)=2\ceil{2\sqrt{x}}$ is a non-decreasing function, we can conclude from (\ref{holesnominp}) that,
\[p(A)\geq p_{min}(n+k)+4k > p_{min}(n)+4k.\]
\end{proof}

We denote this minimum outer perimeter over all polyominoes with $n$ tiles and $m$ holes (whenever such polyominoes exist) by $p_{\min}^{out}(n,m)$.
% where $n$ stands for the number of squares in the polyomino and $m$ for the number of holes. This question can also take into account the area covered by each one of the holes of the polyomino, but for making the arguments easy to follow one can think of each hole as having area 1.
%
Note that we always have,
\begin{eqnarray}\label{pminpminout}
p_{\min}(n+m) \leq p_{\min}^{out}(n,m),
\end{eqnarray}
by definition.
%
%Equality is attained by some polyominoes. For example, the equality holds if $n=7$ and $m=1$; this is also true if $n=14$ and $m=3$ (for the last example see Figure \ref{minpermaxholes}).
%%%%
%\input{minpermaxholes.txt}
%%%%
%In order to show a polyomino for which the equality is not attained, we need an upper bound for the maximum number of holes that polyominoes with the same number of tiles can have. The way we construct an upper bound is based on having a lower bound. In the next section we construct a lower bound for $f(n)$.

%We give an example for which $p_{\min}(n+m) < p_{\min}^{out}(n,m)$ in the last section of this paper.

\subsection{Main upper bound.}\label{SSmainub}

We prove in the next lemma a general result which allows us to generate an upper bound of $f$ from any lower bound. We will apply this lemma again in Sections \ref{FractalSequence} and \ref{Sgeneralb}.

\begin{lemma}\label{LTupperbound}
Let $n$ be any natural number. If $f(n)$ denotes the maximum number of holes that an element of $\mathcal{A}_n$ can have, then, 
\begin{equation}\label{ub}
f(n) \leq \frac{4n-2b_{min}(n)-p_{min}(n+f(n))}{4}.
\end{equation}
\end{lemma}

\begin{proof}
Let $A$ be an element in $\mathcal{A}_{n}$ and let $h(A)$ denote the number of holes in $A$. Then,
\begin{equation}\label{upbound1}
p_{h}(A) = 4n-2b(A)-p_{o}(A)\leq  4n-2b_{min}(n)-p_{o}(A).
\end{equation}

By (\ref{pminpminout}) we have, 
\begin{equation}\label{upbound2}
p_{o}(A)\geq p_{min}^{out}(n,h(A))\geq p_{min}(n+h(A)).
\end{equation}
From inequalities (\ref{upbound1}) and (\ref{upbound2}) we get,
\begin{equation}\label{upbound3}
p_{h}(A) \leq 4n-2b_{min}(n)-p_{min}(n+h(A)).
\end{equation}
Then, if $A$ is a polyomino that has $f(n)$ holes, we get, 
\begin{equation*}
f(n)\leq \frac{p_{h}(A)}{4}\leq \frac{4n-2b_{min}(n)-p_{min}(n+f(n))}{4},
\end{equation*}
which establishes inequality (\ref{ub}). 
\end{proof}

As a corollary for Lemma \ref{LTupperbound}, from any lower bound of $f$, we get the following upper bound for $f$,
\begin{corollary}\label{Cublb}
If $lb_f(n)\leq f(n)$, then
\begin{equation}\label{ublb}
f(n)\leq \frac{1}{2}n-\frac{1}{2}\ceil*{2\sqrt{n+lb_f(n)}}+\frac{1}{2},
\end{equation}
\end{corollary}
\begin{proof}
Let $lb_f(n)\leq f(n)$ for a natural number $n$, then $p_{min}(n+lb_f(n))\leq p_{min}(n+f(n))$, by monotonicity. This inequality allow us to obtain (\ref{ublb}) by substituting $p_{min}(n+f(n))$ with $p_{min}(n+lb_f(n))$ in  (\ref{ub}).
\end{proof}

%\section{A lower bound (fractal construction)}

\section{Polyominoes that attain the maximum number of holes.}\label{FractalSequence} 
\subsection{Construction of the sequence.}
We are going to describe below how to construct a sequence $\{S_{k}\}_{k=1}^\infty$ of polyominoes with $n_k$ squares and $h(S_k) = h_k$ holes. Remember that $n_k$ and $h_k$ were defined in (\ref{Defn}) and (\ref{Defh}) respectively.

The first three elements, $S_1$, $S_2$, and $S_3$ of the sequence were shown in Figure \ref{Fonehole}.

%We generate the second element of the sequence, $A_2$, from $A_1$ by placing two perpendicular lines, parallel to the horizontal and vertical edges of the squares of $A_1$, with intersection at the center of the square that we have removed in $S_1$.

%Then we perform two reflections based on these two lines: the first reflection in respect of the vertical line ; then, on the obtained polyomino, we perform a reflection with regard to the horizontal line.

%We name the obtained polyomino as $S_2$. Finally we remove the top rightmost tile from $S_2$ to obtain $A_2$. This process of creating $A_2$ from $A_1$ is depicted in Figure \ref{Ifractalconstruction}. 

To generate the rest of the sequence for $n \geq 2$, we follow the next general recursion process:

\begin{itemize}
\item First, place a rotation point in the center of the top right tile of $S_{n-1}$.
\item Then, rotate $S_{n-1}$ with respect to this point ninety degrees four times creating four, overlapping copies. 
\item Finally, remove the tile containing the rotation point.
\end{itemize}

%%%%
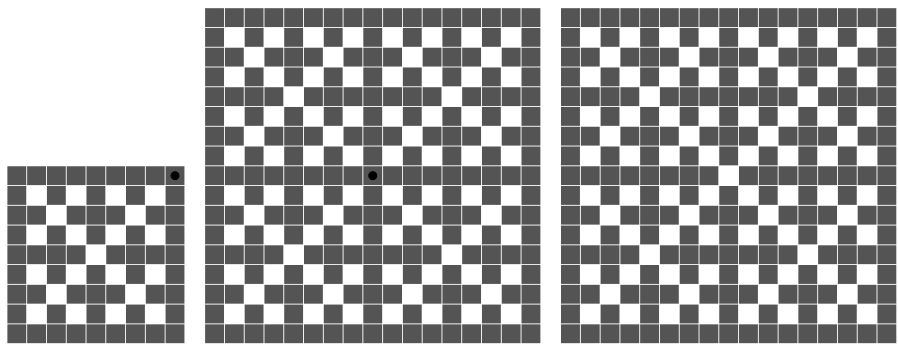
\begin{figure}[ht]
\centering

\begin{tikzpicture} [scale=.45]
\foreach \x/\y in {0 / 0, 0 / 1, 1 / 0, 2 / 0, 2 / 1, 0 / 2, 1 / 2, 4 / 0, 4 / 1, 3 / 0, 4 / 2, 3 / 2, 0 / 4, 0 / 3, 1 / 4, 2 / 4, 2 / 3, 4 / 3, 3 / 4, 8 / 0, 8 / 1, 7 / 0, 6 / 0, 6 / 1, 8 / 2, 7 / 2, 5 / 0, 5 / 2, 8 / 4, 8 / 3, 7 / 4, 6 / 4, 6 / 3, 5 / 4, 0 / 8, 0 / 7, 1 / 8, 2 / 8, 2 / 7, 0 / 6, 1 / 6, 4 / 8, 4 / 7, 3 / 8, 4 / 6, 3 / 6, 0 / 5, 2 / 5, 4 / 5, 8 / 8, 8 / 7, 7 / 8, 6 / 8, 6 / 7, 8 / 6, 7 / 6, 5 / 8, 5 / 6, 8 / 5, 6 / 5 } { 
\path [draw=gray!66!black, fill=gray!66!black] (.5+\x-0.45, .5+\y-0.45)
-- ++(0,.9)
-- ++(.9,0)
-- ++(0,-.9)
--cycle;
}

\foreach \x/\y in {0 / 0, 0 / 1, 1 / 0, 2 / 0, 2 / 1, 0 / 2, 1 / 2, 4 / 0, 4 / 1, 3 / 0, 4 / 2, 3 / 2, 0 / 4, 0 / 3, 1 / 4, 2 / 4, 2 / 3, 4 / 3, 3 / 4, 8 / 0, 8 / 1, 7 / 0, 6 / 0, 6 / 1, 8 / 2, 7 / 2, 5 / 0, 5 / 2, 8 / 4, 8 / 3, 7 / 4, 6 / 4, 6 / 3, 5 / 4, 0 / 8, 0 / 7, 1 / 8, 2 / 8, 2 / 7, 0 / 6, 1 / 6, 4 / 8, 4 / 7, 3 / 8, 4 / 6, 3 / 6, 0 / 5, 2 / 5, 4 / 5, 8 / 7, 7 / 8, 6 / 8, 6 / 7, 8 / 6, 7 / 6, 5 / 8, 5 / 6, 8 / 5, 6 / 5, 16 / 0, 16 / 1, 15 / 0, 14 / 0, 14 / 1, 16 / 2, 15 / 2, 12 / 0, 12 / 1, 13 / 0, 12 / 2, 13 / 2, 16 / 4, 16 / 3, 15 / 4, 14 / 4, 14 / 3, 12 / 3, 13 / 4, 9 / 0, 10 / 0, 10 / 1, 9 / 2, 11 / 0, 11 / 2, 9 / 4, 10 / 4, 10 / 3, 11 / 4, 16 / 8, 16 / 7, 15 / 8, 14 / 8, 14 / 7, 16 / 6, 15 / 6, 12 / 8, 12 / 7, 13 / 8, 12 / 6, 13 / 6, 16 / 5, 14 / 5, 12 / 5, 9 / 8, 10 / 8, 10 / 7, 9 / 6, 11 / 8, 11 / 6, 10 / 5, 0 / 16, 0 / 15, 1 / 16, 2 / 16, 2 / 15, 0 / 14, 1 / 14, 4 / 16, 4 / 15, 3 / 16, 4 / 14, 3 / 14, 0 / 12, 0 / 13, 1 / 12, 2 / 12, 2 / 13, 4 / 13, 3 / 12, 8 / 16, 8 / 15, 7 / 16, 6 / 16, 6 / 15, 8 / 14, 7 / 14, 5 / 16, 5 / 14, 8 / 12, 8 / 13, 7 / 12, 6 / 12, 6 / 13, 5 / 12, 0 / 9, 2 / 9, 0 / 10, 1 / 10, 4 / 9, 4 / 10, 3 / 10, 0 / 11, 2 / 11, 4 / 11, 8 / 9, 6 / 9, 8 / 10, 7 / 10, 5 / 10, 8 / 11, 6 / 11, 16 / 16, 16 / 15, 15 / 16, 14 / 16, 14 / 15, 16 / 14, 15 / 14, 12 / 16, 12 / 15, 13 / 16, 12 / 14, 13 / 14, 16 / 12, 16 / 13, 15 / 12, 14 / 12, 14 / 13, 12 / 13, 13 / 12, 9 / 16, 10 / 16, 10 / 15, 9 / 14, 11 / 16, 11 / 14, 9 / 12, 10 / 12, 10 / 13, 11 / 12, 16 / 9, 14 / 9, 16 / 10, 15 / 10, 12 / 9, 12 / 10, 13 / 10, 16 / 11, 14 / 11, 12 / 11, 10 / 9, 9 / 10, 11 / 10, 10 / 11, 8/8} { 
\path [draw=gray!66!black, fill=gray!66!black] (10+.5+\x-0.45, .5+\y-0.45)
-- ++(0,.9)
-- ++(.9,0)
-- ++(0,-.9)
--cycle;
}

\foreach \x/\y in {0 / 0, 0 / 1, 1 / 0, 2 / 0, 2 / 1, 0 / 2, 1 / 2, 4 / 0, 4 / 1, 3 / 0, 4 / 2, 3 / 2, 0 / 4, 0 / 3, 1 / 4, 2 / 4, 2 / 3, 4 / 3, 3 / 4, 8 / 0, 8 / 1, 7 / 0, 6 / 0, 6 / 1, 8 / 2, 7 / 2, 5 / 0, 5 / 2, 8 / 4, 8 / 3, 7 / 4, 6 / 4, 6 / 3, 5 / 4, 0 / 8, 0 / 7, 1 / 8, 2 / 8, 2 / 7, 0 / 6, 1 / 6, 4 / 8, 4 / 7, 3 / 8, 4 / 6, 3 / 6, 0 / 5, 2 / 5, 4 / 5, 8 / 7, 7 / 8, 6 / 8, 6 / 7, 8 / 6, 7 / 6, 5 / 8, 5 / 6, 8 / 5, 6 / 5, 16 / 0, 16 / 1, 15 / 0, 14 / 0, 14 / 1, 16 / 2, 15 / 2, 12 / 0, 12 / 1, 13 / 0, 12 / 2, 13 / 2, 16 / 4, 16 / 3, 15 / 4, 14 / 4, 14 / 3, 12 / 3, 13 / 4, 9 / 0, 10 / 0, 10 / 1, 9 / 2, 11 / 0, 11 / 2, 9 / 4, 10 / 4, 10 / 3, 11 / 4, 16 / 8, 16 / 7, 15 / 8, 14 / 8, 14 / 7, 16 / 6, 15 / 6, 12 / 8, 12 / 7, 13 / 8, 12 / 6, 13 / 6, 16 / 5, 14 / 5, 12 / 5, 9 / 8, 10 / 8, 10 / 7, 9 / 6, 11 / 8, 11 / 6, 10 / 5, 0 / 16, 0 / 15, 1 / 16, 2 / 16, 2 / 15, 0 / 14, 1 / 14, 4 / 16, 4 / 15, 3 / 16, 4 / 14, 3 / 14, 0 / 12, 0 / 13, 1 / 12, 2 / 12, 2 / 13, 4 / 13, 3 / 12, 8 / 16, 8 / 15, 7 / 16, 6 / 16, 6 / 15, 8 / 14, 7 / 14, 5 / 16, 5 / 14, 8 / 12, 8 / 13, 7 / 12, 6 / 12, 6 / 13, 5 / 12, 0 / 9, 2 / 9, 0 / 10, 1 / 10, 4 / 9, 4 / 10, 3 / 10, 0 / 11, 2 / 11, 4 / 11, 8 / 9, 6 / 9, 8 / 10, 7 / 10, 5 / 10, 8 / 11, 6 / 11, 16 / 16, 16 / 15, 15 / 16, 14 / 16, 14 / 15, 16 / 14, 15 / 14, 12 / 16, 12 / 15, 13 / 16, 12 / 14, 13 / 14, 16 / 12, 16 / 13, 15 / 12, 14 / 12, 14 / 13, 12 / 13, 13 / 12, 9 / 16, 10 / 16, 10 / 15, 9 / 14, 11 / 16, 11 / 14, 9 / 12, 10 / 12, 10 / 13, 11 / 12, 16 / 9, 14 / 9, 16 / 10, 15 / 10, 12 / 9, 12 / 10, 13 / 10, 16 / 11, 14 / 11, 12 / 11, 10 / 9, 9 / 10, 11 / 10, 10 / 11} { 
\path [draw=gray!66!black, fill=gray!66!black] (28.5+\x-0.45, .5+\y-0.45)
-- ++(0,.9)
-- ++(.9,0)
-- ++(0,-.9)
--cycle;
}

\draw [fill=black, black] (8.5,8.5) circle [radius=.2];

\draw [fill=black, black] (18.5, 8.5) circle [radius=.2];
 \end{tikzpicture}
\caption{Generating $S_{n+1}$ from $S_n$ (L to R). (1) The polyomino $S_3$, (2) four overlapping rotated copies of $S_3$, and (3) the polyomino $S_4$ made by removing the tile of rotation.
 }
 \label{Ifractalconstruction}
\end{figure}
%%%%

\subsection{Properties of the sequence $S_k$ and proof of Theorem \ref{Ttremebundoexactly}.}
From this construction, we observe that for $k \ge 1$ we have the recursion $h(S_{k}) = 4h(S_{k-1})+1$. The factor of four is due to the four reflected copies of $S_{k-1}$ generated in the process of constructing $S_k$. The one hole added is generated by the square removed after the rotation process. Then, because $h(S_1)=1$, we get $h(S_k)= h_k$ for all $k\geq 1$.

Let $s_k$ be the number of tiles in $S_k$ for $k\geq 1$. The sequence $s_k$ satisfies the recursion, $$s_{k+1}=4s_k-4(2^k+1),$$ because the polyominoes $S_k$ have side lengths of $2^k +1$ tiles and, in the rotation process, $4(2^k+1)$ tiles overlap. Additionally, the sequence $n_k$ satisfies the relationship, $$n_{k+1}=4n_k-4(2^k+1).$$ Then, because both $n_k$ and $s_k$ satisfy the same recursion relationship and are equal in the first element ($s_1=8$, $n_1=8$), we can conclude that $n_k$ and $s_k$ are the same sequences.

We have proved the following lemma.
\begin{lemma}\label{Lsequencek}
There exists a sequence of polyominoes $\{S_k\}_{k=1}^\infty$, such that $S_k$ has $n_k$ tiles and $h_k$ holes.
\end{lemma}
%
%The rest of this section is the proof of Theorem \ref{Ttremebundoexactly}.

\begin{proof}[Proof of Theorem \ref{Ttremebundoexactly}]
First we show that $f(n_k)=h_k$. From Lemma \ref{Lsequencek} we know, 
\begin{equation}\label{eq}
h_k \leq f(n_k).
\end{equation} Substituting this lower bound in  (\ref{ublb}) we have, 

\begin{equation}\label{pluginmaxexactly1}
f(n_k) \leq \frac{1}{2}n_k-\frac{1}{2}\ceil*{2\sqrt{n_k + h_k} } + \frac{1}{2}.
\end{equation}
From the easily verified identity, 
\[h_k +\frac{1}{2} = \frac{1}{2}n_k-\frac{1}{2}\ceil*{2\sqrt{n_k+h_k}}+\frac{1}{2},
\] and inequalities (\ref{pluginmaxexactly1}) and (\ref{eq}), we get, \begin{equation}
h_k  \leq f(n_k) \leq h_k + \frac{1}{2}.
\end{equation}
This implies that, 
\begin{equation}\label{Efnkhk}
f(n_k) = h_k,
\end{equation}
because $f(n_k)$ and $h_k$ are integers.

Now, we prove that $f(n_k-1) = h_k$. By removing the upper leftmost square from each $S_k$, it is possible to generate a sequence of polyominoes $\{A_k\}_{k=1}^\infty$ with $n_k-1$ tiles each $A_k$, such that $h(A_k) = h_k$. This implies that $h_k\leq f(n_k-1).$
Then, because $f(n_k-1) \leq f(n_k)$ and $f(n_k)=h_k$, we can conclude that,
 \[f(n_k-1)=h_k. \]

%In an analogous way used to prove that $f(n_k)=h_k$, we will prove that inequality $f(n_k-1)\leq h_k$ holds as a consequence of  the next lemma.
%\begin{lemma}
%For $k \geq 1$, we have
%\begin{equation}\label{pluginmaxexactly2}
%h_k = \frac{1}{2}(n_k-1)-\frac{1}{2}\ceil*{2\sqrt{(n_k-1)+h_k}}+\frac{1}{2}.
%\end{equation}
%\end{lemma}
%
%\begin{proof}
%Proving equality (\ref{pluginmaxexactly2}) is equivalent to showing that,
%\begin{equation}\label{insteadofpluginmaxexactly}
%2^k + 1 = \frac{1}{2}\ceil*{2\sqrt{4^k+2^{k+1}}},
%\end{equation}
%for any $k \geq 1$.
%Define, \[\epsilon_k = 2^k+1- \sqrt{4^k+2^{k+1}}.\] A little calculus shows that $\epsilon_k$ is decreasing for $k \geq 1$ and $\epsilon_1 <1/4$. This implies that $\epsilon_k < 1/4$ for every $k \ge 1$.
%From this we get,
%\begin{equation}
%2(2^k+1)-\frac{1}{2}<2\sqrt{4^k+ 2^{k+1}} < 2(2^k+1),
%\end{equation}
%from where,
%\[
%\ceil*{2\sqrt{4^k+ 2^{k+1}}}=2(2^k+1),\]
%which proves (\ref{insteadofpluginmaxexactly}).
%\end{proof}
%Substituting in (\ref{ublb}) the lower bound $lb_f(n_k-1) = h_k$, we get,
%\begin{equation}
%f(n_k-1)\leq \frac{1}{2}(n_k-1)-\frac{1}{2}\ceil*{2\sqrt{(n_k-1)+h_k}}+\frac{1}{2}.
%\end{equation}
%From this inequality and equality (\ref{pluginmaxexactly2})  we conclude that $f(n_k-1) = h_k$.

Finally, we prove that $f(n_k-2) = h_k-1$. 
Because $f(n_k-1) = h_k$ and $f$ is a non decreasing function, it is forced that $f(n_k-2) \leq  h_k$. If we assume by way of contradiction that $f(n_k-2) = h_k$, then, using (\ref{ublb}) it must be the case that,
\begin{equation}
h_k=f(n_k-2)\leq \frac{1}{2}(n_k-2)-\frac{1}{2}\ceil*{2\sqrt{(n_k-2)+h_k}}+\frac{1}{2}.
\end{equation}
Substituting (\ref{Defn}) and (\ref{Defh}) in this inequality leads to a contradiction. Then $f(n_k-2)< h_k$; and, from this inequality and Lemma \ref{Ljumpsoff}, we can conclude that,
\begin{equation}
f(n_k-2) = h_k - 1.
\end{equation}

\end{proof}

\section{ General bounds}\label{Sgeneralb}
In this section we prove Theorem \ref{Ttremebundo}. We first prove the lower bound. We construct a sequence of polyominoes $\{R_k\}_{k=1}^\infty$ of polyominoes with  $m_k=40k^2+20k$ tiles and $t_k=20k^2$ holes, as follows.

We first place $10k^2$ copies of the pattern $S$ (in Figure \ref{pattern}) into a rectangle $6k$ high and $10k$ long. We add a top row of  $10k$ tiles, and a leftmost column of $6k-1$  tiles. Finally, we attach $2k$ vertically aligned dominoes (for a total of $4k$ tiles) to the rightmost column. The polyomino $R_2$ is depicted in Figure \ref{Fm2}, and the $40k^2$ tiles just described are colored with the lightest gray in this figure. The initial, repeated pattern $S$ is bordered in black within $R_2$. 

Next ``we fill in the gaps'' between these constructions to a family of polyominoes $\{R_{k,l} \}$ with $m_{k,l} = 40k^2 + 20k + l$ tiles and $t_{k,l} = 20k^2 + \lfloor l / 2 \rfloor$ holes, defined whenever
$$0 \le l \le \sum_{i=1}^{2k-1} 2i = 2k(2k-1).$$
The polyomino $R_{k,l}$ is constructed from $R_{k}$ by adding tiles along the right side, continuing the domino pattern, as in Figure \ref{m1}. Every two tiles added creates one hole.

Note that $m_{k+1} - m_k  = 80k+ 60$, and $$80k + 60 \le 2k(2k-1)$$ for $k \ge 42$.
We define a sequence of polyominoes $\{ R'_n \}$ for all $n \geq m_{42}$, as follows.
Let $k$ be the largest number such that $m_k \le n$, let $l = n - m_k$, and then define $R'_n = R_{k,l}$.

%We note that since  $m_{k+1} - m_{k} = 80k + 60$, we have that $l \ge 80k + 60$.

Now we check that if $C_1 > \sqrt{5/2}$, then $f(n) \ge n/2 - C_1 \sqrt{n}$ for all large enough $n$.

The polyomino $R'_n= R_{k,l}$ has $m_{k,l}$ tiles and $t_{k,l}$ holes. Since $k$ and $l$ are nonnegative, we have
\begin{equation}
(5/2) \left(40k^2+ 20k+ l \right) = 100k^2 + 50k + (5/2) l \ge 100 k^2 + 20 k + 1.
\end{equation}
Taking square roots of both sides gives 
\begin{equation}
\sqrt{5/2} \sqrt{40k^2 +20k+l } \ge 10 k+1.\\
\end{equation}

Since $C_1 > \sqrt{5/2}$, this gives
\begin{equation} \label{eqn:40k2}
C_1 \sqrt{40k^2 +20k+l } \ge 10 k+1.\\
\end{equation}

Then 
\begin{equation} \label{eqn:20k2}
20k^2 + \left\lfloor \frac{l}{2} \right\rfloor \ge \frac{1}{2} \left( 40 k^2 + 20k + l \right)  - C_1  \sqrt{40k^2 +20k+l }.
\end{equation}

Inequality (\ref{eqn:20k2}) was obtained from (\ref{eqn:40k2}) by adding $20k^2$ to both sides, adding  the inequality $\left\lfloor \frac{l}{2} \right\rfloor  \ge \frac{l}{2} - 1$, and rearranging terms.

By considering the sequence $R'_n$, we have
\begin{equation} \label{eqn:mainlower}
f(n) \ge \frac{n}{2} - C_1 \sqrt{n}
\end{equation} for all large enough $n$, as desired.

Combining inequality (\ref{eqn:mainlower}) with Corollary \ref{Cublb}, we get an upper bound that allows us to complete the proof of Theorem  \ref{Ttremebundo}. Indeed, we have
\begin{align}
f(n) & \le \frac{1}{2} n -\frac{1}{2} \left\lceil 2 \sqrt{n + \frac{n}{2} - C_1 \sqrt{n}} \right\rceil + \frac{1}{2} \\
& \le  \frac{1}{2} n - \sqrt{\frac{3n}{2} - C_1 \sqrt{n} } + 1 \\
& \le \frac{1}{2} n - C_2 \sqrt{n}, \label{eqn:mainupper}
\end{align}
for $C_2 < \sqrt{3/2}$ and large enough $n$.

%%%%%%%%%%%%%%

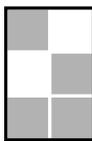
\begin{figure}[ht]
\centering

\begin{tikzpicture}
\foreach \x/\y in {0 / 0, 1 / 0, 1 / 1, 0/2 } { 
\path [draw=lightgray!93!black, fill=lightgray!93!black] (.5+\x-0.45, .5+\y-0.45)
-- ++(0,.9)
-- ++(.9,0)
-- ++(0,-.9)
--cycle;
}

\draw[black, very thick] (0,0)rectangle (2,3);

 \end{tikzpicture}
 \caption{The pattern $S$ consists of $4$ tiles and $2$ holes.}
 \label{pattern}
\end{figure}

\begin{figure}[ht]
\centering

\begin{tikzpicture} [scale=0.6]
\foreach \x/\y in {21 / 3, 21 / 4, 21 / 6, 21 / 7, 21 / 9, 21 / 10, 21 / 12, 21 / 13 } { 
\path [draw=gray!60!black, fill=gray!60!black] (\x-0.45, .5+\y-0.45)
-- ++(0,.9)
-- ++(.9,0)
-- ++(0,-.9)
--cycle;
}

\foreach \x/\y in {0 / 1, 0 / 2, 0 / 3, 0 / 4, 0 / 5, 0 / 6, 0 / 7, 0 / 8, 0 / 9, 0 / 10, 0 / 11, 0 / 12, 1 / 13, 2 / 13, 3 / 13, 4 / 13, 5 / 13, 6 / 13, 7 / 13, 8 / 13, 9 / 13, 10 / 13, 11 / 13, 12 / 13, 13 / 13, 14 / 13, 15 / 13, 16 / 13, 17 / 13, 18 / 13, 19 / 13, 20 / 13} { 
\path [draw=gray, fill=gray] (\x-0.45, .5+\y-0.45)
-- ++(0,.9)
-- ++(.9,0)
-- ++(0,-.9)
--cycle;
}

\foreach \x/\y in {1 / 1, 2 / 1, 3 / 1, 4 / 1, 5 / 1, 6 / 1, 7 / 1, 8 / 1, 9 / 1, 10 / 1, 11 / 1, 12 / 1, 13 / 1, 14 / 1, 15 / 1, 16 / 1, 17 / 1, 18 / 1, 19 / 1, 20 / 1, 2 / 2, 4 / 2, 6 / 2, 8 / 2, 10 / 2, 12 / 2, 14 / 2, 16 / 2, 18 / 2, 20 / 2, 1 / 3, 3 / 3, 5 / 3, 7 / 3, 9 / 3, 11 / 3, 13 / 3, 15 / 3, 17 / 3, 19 / 3, 1 / 4, 2 / 4, 3 / 4, 4 / 4, 5 / 4, 6 / 4, 7 / 4, 8 / 4, 9 / 4, 10 / 4, 11 / 4, 12 / 4, 13 / 4, 14 / 4, 15 / 4, 16 / 4, 17 / 4, 18 / 4, 19 / 4, 20 / 4, 2 / 5, 4 / 5, 6 / 5, 8 / 5, 10 / 5, 12 / 5, 14 / 5, 16 / 5, 18 / 5, 20 / 5, 1 / 6, 3 / 6, 5 / 6, 7 / 6, 9 / 6, 11 / 6, 13 / 6, 15 / 6, 17 / 6, 19 / 6, 1 / 7, 2 / 7, 3 / 7, 4 / 7, 5 / 7, 6 / 7, 7 / 7, 8 / 7, 9 / 7, 10 / 7, 11 / 7, 12 / 7, 13 / 7, 14 / 7, 15 / 7, 16 / 7, 17 / 7, 18 / 7, 19 / 7, 20 / 7, 2 / 8, 4 / 8, 6 / 8, 8 / 8, 10 / 8, 12 / 8, 14 / 8, 16 / 8, 18 / 8, 20 / 8, 1 / 9, 3 / 9, 5 / 9, 7 / 9, 9 / 9, 11 / 9, 13 / 9, 15 / 9, 17 / 9, 19 / 9, 1 / 10, 2 / 10, 3 / 10, 4 / 10, 5 / 10, 6 / 10, 7 / 10, 8 / 10, 9 / 10, 10 / 10, 11 / 10, 12 / 10, 13 / 10, 14 / 10, 15 / 10, 16 / 10, 17 / 10, 18 / 10, 19 / 10, 20 / 10, 2 / 11, 4 / 11, 6 / 11, 8 / 11, 10 / 11, 12 / 11, 14 / 11, 16 / 11, 18 / 11, 20 / 11, 1 / 12, 3 / 12, 5 / 12, 7 / 12, 9 / 12, 11 / 12, 13 / 12, 15 / 12, 17 / 12, 19 / 12} { 
\path [draw=lightgray!93!black, fill=lightgray!93!black] (\x-0.45, .5+\y-0.45)
-- ++(0,.9)
-- ++(.9,0)
-- ++(0,-.9)
--cycle;
}

\draw[black, very thick] (.5,1)rectangle (20.5 ,13);
\draw[black, very thick] (.5,4)-- (20.5 ,4);
\draw[black, very thick] (.5,7)-- (20.5 ,7);
\draw[black, very thick] (.5,10)-- (20.5 ,10);

\draw[black, very thick] (2.5,1)-- (2.5 ,13);
\draw[black, very thick] (4.5,1)-- (4.5 ,13);
\draw[black, very thick] (6.5,1)-- (6.5 ,13);
\draw[black, very thick] (8.5,1)-- (8.5 ,13);
\draw[black, very thick] (10.5,1)-- (10.5 ,13);
\draw[black, very thick] (12.5,1)-- (12.5 ,13);
\draw[black, very thick] (14.5,1)-- (14.5 ,13);
\draw[black, very thick] (16.5,1)-- (16.5 ,13);
\draw[black, very thick] (18.5,1)-- (18.5 ,13);

 \end{tikzpicture}
 \caption{The polyomino $R_2$ has $200$ tiles and $80$ holes.}
 \label{Fm2}
\end{figure}
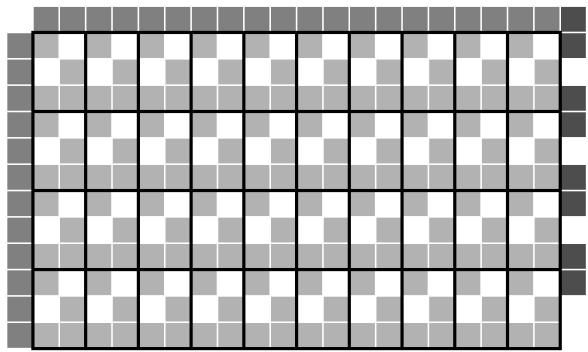
\begin{figure}[ht]
\centering

\begin{tikzpicture} [scale=.6] 
\foreach \x/\y in {32 / 5, 32 / 6, 33 / 7, 32 / 8, 33 / 8, 32 / 9, 34 / 9, 33 / 10, 34 / 10, 32 / 11, 33 / 11, 32 / 12, 34 / 12, 33 / 13, 34 / 13, 32 / 14, 33 / 14, 32 / 15, 34 / 15, 33 / 16, 34 / 16, 32 / 17, 33 / 17, 32 / 18 } { 
\path [draw=gray!60!black, fill=gray!60!black] (\x-0.45, .5+\y-0.45)
-- ++(0,.9)
-- ++(.9,0)
-- ++(0,-.9)
--cycle;
}

\foreach \x/\y in {0 / 1, 1 / 1, 2 / 1, 3 / 1, 4 / 1, 5 / 1, 6 / 1, 7 / 1, 8 / 1, 9 / 1, 10 / 1, 11 / 1, 12 / 1, 13 / 1, 14 / 1, 15 / 1, 16 / 1, 17 / 1, 18 / 1, 19 / 1, 20 / 1, 21 / 1, 22 / 1, 23 / 1, 24 / 1, 25 / 1, 26 / 1, 27 / 1, 28 / 1, 29 / 1, 30 / 1, 0 / 2, 2 / 2, 4 / 2, 6 / 2, 8 / 2, 10 / 2, 12 / 2, 14 / 2, 16 / 2, 18 / 2, 20 / 2, 22 / 2, 24 / 2, 26 / 2, 28 / 2, 30 / 2, 0 / 3, 1 / 3, 3 / 3, 5 / 3, 7 / 3, 9 / 3, 11 / 3, 13 / 3, 15 / 3, 17 / 3, 19 / 3, 21 / 3, 23 / 3, 25 / 3, 27 / 3, 29 / 3, 31 / 3, 0 / 4, 1 / 4, 2 / 4, 3 / 4, 4 / 4, 5 / 4, 6 / 4, 7 / 4, 8 / 4, 9 / 4, 10 / 4, 11 / 4, 12 / 4, 13 / 4, 14 / 4, 15 / 4, 16 / 4, 17 / 4, 18 / 4, 19 / 4, 20 / 4, 21 / 4, 22 / 4, 23 / 4, 24 / 4, 25 / 4, 26 / 4, 27 / 4, 28 / 4, 29 / 4, 30 / 4, 31 / 4, 0 / 5, 2 / 5, 4 / 5, 6 / 5, 8 / 5, 10 / 5, 12 / 5, 14 / 5, 16 / 5, 18 / 5, 20 / 5, 22 / 5, 24 / 5, 26 / 5, 28 / 5, 30 / 5, 0 / 6, 1 / 6, 3 / 6, 5 / 6, 7 / 6, 9 / 6, 11 / 6, 13 / 6, 15 / 6, 17 / 6, 19 / 6, 21 / 6, 23 / 6, 25 / 6, 27 / 6, 29 / 6, 31 / 6, 0 / 7, 1 / 7, 2 / 7, 3 / 7, 4 / 7, 5 / 7, 6 / 7, 7 / 7, 8 / 7, 9 / 7, 10 / 7, 11 / 7, 12 / 7, 13 / 7, 14 / 7, 15 / 7, 16 / 7, 17 / 7, 18 / 7, 19 / 7, 20 / 7, 21 / 7, 22 / 7, 23 / 7, 24 / 7, 25 / 7, 26 / 7, 27 / 7, 28 / 7, 29 / 7, 30 / 7, 31 / 7, 0 / 8, 2 / 8, 4 / 8, 6 / 8, 8 / 8, 10 / 8, 12 / 8, 14 / 8, 16 / 8, 18 / 8, 20 / 8, 22 / 8, 24 / 8, 26 / 8, 28 / 8, 30 / 8, 0 / 9, 1 / 9, 3 / 9, 5 / 9, 7 / 9, 9 / 9, 11 / 9, 13 / 9, 15 / 9, 17 / 9, 19 / 9, 21 / 9, 23 / 9, 25 / 9, 27 / 9, 29 / 9, 31 / 9, 0 / 10, 1 / 10, 2 / 10, 3 / 10, 4 / 10, 5 / 10, 6 / 10, 7 / 10, 8 / 10, 9 / 10, 10 / 10, 11 / 10, 12 / 10, 13 / 10, 14 / 10, 15 / 10, 16 / 10, 17 / 10, 18 / 10, 19 / 10, 20 / 10, 21 / 10, 22 / 10, 23 / 10, 24 / 10, 25 / 10, 26 / 10, 27 / 10, 28 / 10, 29 / 10, 30 / 10, 31 / 10, 0 / 11, 2 / 11, 4 / 11, 6 / 11, 8 / 11, 10 / 11, 12 / 11, 14 / 11, 16 / 11, 18 / 11, 20 / 11, 22 / 11, 24 / 11, 26 / 11, 28 / 11, 30 / 11, 0 / 12, 1 / 12, 3 / 12, 5 / 12, 7 / 12, 9 / 12, 11 / 12, 13 / 12, 15 / 12, 17 / 12, 19 / 12, 21 / 12, 23 / 12, 25 / 12, 27 / 12, 29 / 12, 31 / 12, 0 / 13, 1 / 13, 2 / 13, 3 / 13, 4 / 13, 5 / 13, 6 / 13, 7 / 13, 8 / 13, 9 / 13, 10 / 13, 11 / 13, 12 / 13, 13 / 13, 14 / 13, 15 / 13, 16 / 13, 17 / 13, 18 / 13, 19 / 13, 20 / 13, 21 / 13, 22 / 13, 23 / 13, 24 / 13, 25 / 13, 26 / 13, 27 / 13, 28 / 13, 29 / 13, 30 / 13, 31 / 13, 0 / 14, 2 / 14, 4 / 14, 6 / 14, 8 / 14, 10 / 14, 12 / 14, 14 / 14, 16 / 14, 18 / 14, 20 / 14, 22 / 14, 24 / 14, 26 / 14, 28 / 14, 30 / 14, 0 / 15, 1 / 15, 3 / 15, 5 / 15, 7 / 15, 9 / 15, 11 / 15, 13 / 15, 15 / 15, 17 / 15, 19 / 15, 21 / 15, 23 / 15, 25 / 15, 27 / 15, 29 / 15, 31 / 15, 0 / 16, 1 / 16, 2 / 16, 3 / 16, 4 / 16, 5 / 16, 6 / 16, 7 / 16, 8 / 16, 9 / 16, 10 / 16, 11 / 16, 12 / 16, 13 / 16, 14 / 16, 15 / 16, 16 / 16, 17 / 16, 18 / 16, 19 / 16, 20 / 16, 21 / 16, 22 / 16, 23 / 16, 24 / 16, 25 / 16, 26 / 16, 27 / 16, 28 / 16, 29 / 16, 30 / 16, 31 / 16, 0 / 17, 2 / 17, 4 / 17, 6 / 17, 8 / 17, 10 / 17, 12 / 17, 14 / 17, 16 / 17, 18 / 17, 20 / 17, 22 / 17, 24 / 17, 26 / 17, 28 / 17, 30 / 17, 0 / 18, 1 / 18, 3 / 18, 5 / 18, 7 / 18, 9 / 18, 11 / 18, 13 / 18, 15 / 18, 17 / 18, 19 / 18, 21 / 18, 23 / 18, 25 / 18, 27 / 18, 29 / 18, 31 / 18, 1 / 19, 2 / 19, 3 / 19, 4 / 19, 5 / 19, 6 / 19, 7 / 19, 8 / 19, 9 / 19, 10 / 19, 11 / 19, 12 / 19, 13 / 19, 14 / 19, 15 / 19, 16 / 19, 17 / 19, 18 / 19, 19 / 19, 20 / 19, 21 / 19, 22 / 19, 23 / 19, 24 / 19, 25 / 19, 26 / 19, 27 / 19, 28 / 19, 29 / 19, 30 / 19, 31 / 19} { 
\path [draw=gray, fill=gray] (\x-0.45, .5+\y-0.45)
-- ++(0,.9)
-- ++(.9,0)
-- ++(0,-.9)
--cycle;
}

 \end{tikzpicture} 
\caption{$R_3$ has $420$ tiles and $180$ holes. Adding tiles in a domino pattern on the right side yields  $R_{3,24}$, which has $444$ tiles and $192$ holes.}
 \label{m1}
\end{figure}
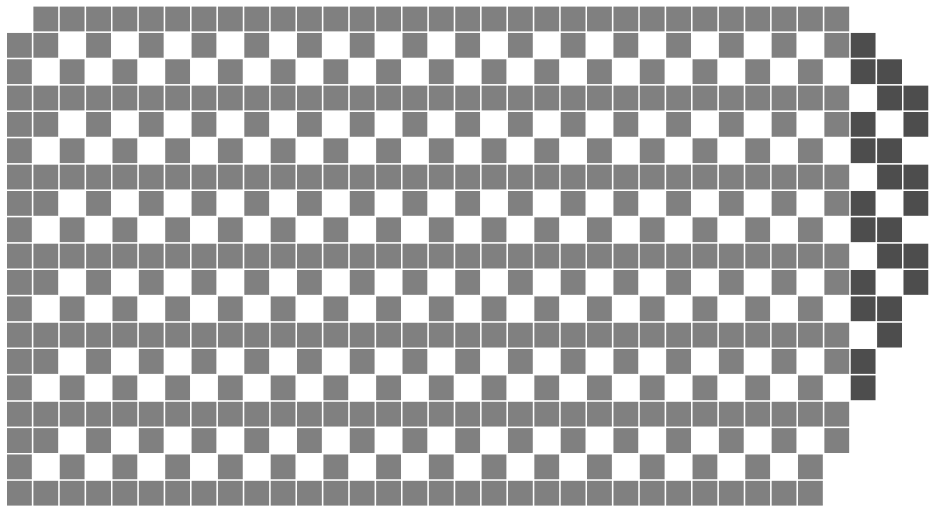

\section{Concluding remarks}

Theorem \ref{Ttremebundoexactly} makes us guess that the upper bound in Theorem \ref{Ttremebundo} is correct, and that the lower bound might be improved to the following.

\begin{conjecture} For every $C_{1}>\sqrt{3 / 2}$ there exists an $n_0=n_0(C_1)$ such that
$$f(n) \ge \frac{1}{2} n - C_1 \sqrt{n}.$$
for all $n \ge n_0$. \end{conjecture}
It might even be possible to find an exact formula for $f$. The sequence of polyominoes $\{ S_k \}$ shows that
$$f(n) =\frac{1}{2} n - \sqrt{ \frac{3}{2} n +\frac{1}{4}} + \frac{1}{2},$$
infinitely often.

The recursive construction suggests that the main sequence of polyominoes $\{ S_k \}$ is approaching some limiting fractal shape. Elliot Paquette pointed out to us that one way to make sense of this idea is to consider the ``inner boundary'' of a polyomino $S_n$ to be an immersed circle in $\mathbb{R}^2$. Appropriately rescaling and reparameterizing, these circles seem to converge to a space-filling curve in $[0,1]^2$.

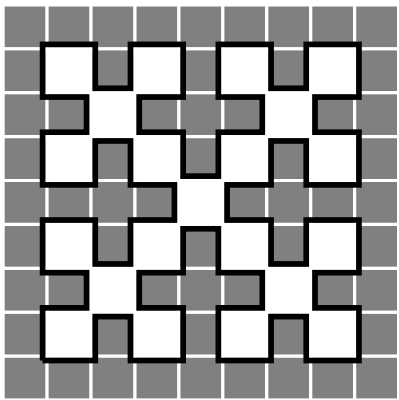
\begin{figure}[ht]
\centering
\begin{tikzpicture} 
\foreach \x/\y in {0 / 0, 0 / 1, 1 / 0, 2 / 0, 2 / 1, 0 / 2, 1 / 2, 4 / 0, 4 / 1, 3 / 0, 4 / 2, 3 / 2, 0 / 4, 0 / 3, 1 / 4, 2 / 4, 2 / 3, 4 / 3, 3 / 4, 8 / 0, 8 / 1, 7 / 0, 6 / 0, 6 / 1, 8 / 2, 7 / 2, 5 / 0, 5 / 2, 8 / 4, 8 / 3, 7 / 4, 6 / 4, 6 / 3, 5 / 4, 0 / 8, 0 / 7, 1 / 8, 2 / 8, 2 / 7, 0 / 6, 1 / 6, 4 / 8, 4 / 7, 3 / 8, 4 / 6, 3 / 6, 0 / 5, 2 / 5, 4 / 5, 8 / 8, 8 / 7, 7 / 8, 6 / 8, 6 / 7, 8 / 6, 7 / 6, 5 / 8, 5 / 6, 8 / 5, 6 / 5 } { 
\path [draw=gray, fill=gray] (\x+0.05, \y+0.05)
-- ++(0,.9)
-- ++(.9,0)
-- ++(0,-.9)
--cycle;
}

\draw [line width=0.75mm](0.9,0.9)--(2.1,0.9)--(2.1,1.9)--(2.9,1.9)--(2.9,0.9)--(4.1,0.9)--(4.1,2.1)--(3.1,2.1)--(3.1,2.9)--(4.1,2.9)--(4.1,3.9)--(4.9,3.9)--(4.9,2.9)--(5.9,2.9)--(5.9,2.1)--(4.9,2.1)--(4.9,0.9)--(6.1,0.9)--(6.1,1.9)--(6.9,1.9)--(6.9,0.9)--(8.1,0.9)--(8.1,2.1)--(7.1,2.1)--(7.1,2.9)--(8.1,2.9)--(8.1,4.1)--(6.9,4.1)--(6.9,3.1)--(6.1,3.1)--(6.1,4.1)--(5.1,4.1)--(5.1,4.9)--(6.1,4.9)--(6.1,5.9)--(6.9,5.9)--(6.9,4.9)--(8.1,4.9)--(8.1,6.1)--(7.1,6.1)--(7.1,6.9)--(8.1,6.9)--(8.1,8.1)--(6.9,8.1)--(6.9,7.1)--(6.1,7.1)--(6.1,8.1)--(4.9,8.1)--(4.9,6.9)--(5.9,6.9)--(5.9,6.1)--(4.9,6.1)--(4.9,5.1)--(4.1,5.1)--(4.1,6.1)--(3.1,6.1)--(3.1,6.9)--(4.1,6.9)--(4.1,8.1)--(2.9,8.1)--(2.9,7.1)--(2.1,7.1)--(2.1,8.1)--(0.9,8.1)--(0.9,6.9)--(1.9,6.9)--(1.9,6.1)--(0.9,6.1)--(0.9,4.9)--(2.1,4.9)--(2.1,5.9)--(2.9,5.9)--(2.9,4.9)--(3.9,4.9)--(3.9,4.1)--(2.9,4.1)--(2.9,3.1)--(2.1,3.1)--(2.1,4.1)--(0.9,4.1)--(0.9,2.9)--(1.9,2.9)--(1.9,2.1)--(0.9,2.1)--(0.9,0.9);
 \end{tikzpicture}

\caption{The inner boundary is an immersed circle.}
\label{fig:circ}
\end{figure}

We noticed also that one can derive an aperiodic tiling of the plane from our main sequence, as follows. The `planar dual' of a polyomino $S_k$ is a planar graph with one vertex for every tile and one bounded face for every hole---see Figure \ref{Tess}. One can find a limiting infinite polyomino by centering every $S_k$ at the origin, and taking the union of all of them. The planar dual of this infinite polyomino is an aperiodic tiling by squares, pentagons, and hexagons.  
%%%%
\input{tessellation.txt}
%%%%
%%%
%\begin{figure}
%\begin{centering}
%\includegraphics[width=3in]{7deep.jpg}
%\caption{The polyomino $S_7$.}
%\label{fig:fractal}
%\end{centering}
%\end{figure}
%%%
%Acknowledgemts, References and include{Bibliography} in finalpart.txt

\noindent

\section*{Acknowledgments.}
We thank CIMAT (Centro de Investigaci\'on en Matem\'aticas) for hosting us in Summer 2016 and ICERM (Institute for Computational and Experimental Research in Mathematics) for hosting us in Autumn 2016. This work was also supported in part by NSF DMS-1352386.

%\begin{thebibliography}{1}
%\bibitem{Golomb} Golomb Solomon W, Checker boards and polyominoes \textit{The American Mathematical Monthly} \textbf{61} (1954) 675--682.
%\end{thebibliography}

\bibliographystyle{plain}
\bibliography{bibliography}

\end{document}